\documentclass[11pt,a4paper,twosides]{amsart}
\usepackage{amssymb,amsmath,amsthm,mathrsfs,graphicx,xcolor}
\allowdisplaybreaks
\theoremstyle{plain}

\newcommand{\rre}{\mathbb{R}}
\newcommand{\pt}{\partial}

\numberwithin{equation}{section}

\newtheorem{thm}{Theorem}[section]

\newtheorem{lem}[thm]{Lemma}
\newtheorem{prop}[thm]{Proposition}

\theoremstyle{remark}
\newtheorem{rem}{Remark}

\makeatletter
\@addtoreset{equation}{section}
\makeatother

\title[Zakharov-Kuznetsov equation in 3D] 
{Scattering problem for Zakharov-Kuznetsov\\
equation in three space dimensions}

\author{Jun-ichi Segata}

\address{Faculty of Mathematics, Kyushu University, 
Fukuoka, 819-0395, Japan}
\email{segata@math.kyushu-u.ac.jp}


\begin{document}

\maketitle



\begin{abstract}
This paper is a continuation of our previous study \cite{S} 
on the scattering problem for the 
Zakharov-Kuznetsov equation (ZK). 
When the space dimension is three, we construct a 
global solution to (ZK) which scatters to a given free solution 
without smallness assumption on the asymptotic states. 

%

\end{abstract}

%
%
\section{Introduction}

This paper is a continuation of our previous study \cite{S} 
on the scattering problem for the Zakharov-Kuznetsov equation. 
In this paper we focus on the Zakharov-Kuznetsov equation 
in three dimensions:
\begin{equation} 
\pt_tu+\pt_{x_1}\Delta u=\pt_{x_1}(u^2), \qquad(t,x)\in\rre\times\rre^3,
\label{ZK} 
\end{equation} 
where
$u:\rre\times\rre^3\to\mathbb{R}$ is an unknown
function, $x=(x_1,x_2,x_3)$ and $\Delta$ is Laplacian on $\rre^3$. 
Equation (\ref{ZK}) was 
derived by Zakharov-Kuznetsov \cite{ZK} to describe 
unidirectional wave propagation in a magnetized plasma. 
Note that Laedke-Spatschek \cite{LSp} derived (\ref{ZK}) 
from the basic hydrodynamical equations. Furthermore, 
Lannes-Linares-Saut \cite{LLS} gave the rigorous justification 
of (\ref{ZK}) from the Euler-Poisson system for a uniformly 
magnetized media. 

Equation (\ref{ZK}) has the conservation of mass : for any $t\in\rre$, 
\begin{equation} \label{mass}
M[u](t):=\frac12
\int_{\rre^3}u(t,x)^2dx=M[u](0),
\end{equation}
and the conservation of energy : for any $t\in\rre$, 
\begin{eqnarray} 
\ \ \ E[u](t):=\frac12
\int_{\rre^3}|(\nabla u)(t,x)|^2dx
-\frac16\int_{\rre^3}u(t,x)^3 dxdy=E[u](0).
\label{energy}
\end{eqnarray}

The Zakharov-Kuznetsov equation on $\rre^d$ :
\begin{equation} 
\pt_tu+\pt_{x_1}\Delta u=\pt_{x_1}(u^2), \qquad(t,x)\in\rre\times\rre^d
\label{gZK} 
\end{equation} 
%
has been studied from the 
point of view of well-posedness 
\cite{F1,GH,HK1,HK2,K,LS,MP,RV}, 
and stability of soliton \cite{CMPS,D,FHRY1,KSS,Y} 
etc. Concerning the scattering problem for (\ref{gZK}), 
from the fact that the solution of the linear equation 
associated with (\ref{gZK}) decays like $O(t^{-d/2})$ in $L^{\infty}$ 
as $t\to\infty$ 
(see \cite[Theorem 3.2]{KPV} for instance), 
and from the point of view of the linear scattering theory 
(see \cite{RS} for instance), 
we expect that if $d\ge3$, then (at least small) solution to (\ref{gZK}) 
scatters to the free solution. 
Herr-Kinoshita \cite{HK1} 
proved the small data scattering for the initial value problem of 
(\ref{gZK}) with $d\ge5$ in the scaling critical Sobolev space. 
Furthermore, they proved the 
scattering for (\ref{gZK}) with $d=4$ when the initial data is small 
and radial in the last $(d-1)$ variables. 
For two dimensional case $d=2$, the author \cite{S} 
proved the existence of small global solutions to (\ref{gZK}) 
which scatters to a given free solution. 
See also Farah-Linares-Pastor \cite{FLP}, 
Anjolras \cite{A}, and Correia-Kinoshita \cite{CK} for the scattering 
results on (\ref{gZK}) with $d=2$ and power type nonlinearity 
with degree higher than two. 

In this paper we consider the scattering problem for 
(\ref{gZK}) with physically important case $d=3$ 
in the framework of the final state problem. 
To state the our main theorem, 
we introduce several notation. For $0<\delta<1$, 
we define a semi-normed space $(X_{\delta},\|\cdot\|_{X_{\delta}}$) by
\begin{eqnarray}
X_{\delta}&:=&
\{f\in{{\mathcal S}}(\rre^3) ; \|f\|_{X_{\delta}}<\infty\},
\label{normX}\\
\|f\|_{X_{\delta}}&:=&
\||\pt_{x_1}|^{-\delta}\langle x\rangle f\|_{W_x^{6,1}}
+\|(3\pt_{x_1}^2-\pt_{x_2}^2-\pt_{x_3}^2)^{-2}\langle x\rangle f\|_{H_x^7}
\nonumber\\
& &+\|(3\pt_{x_1}^2-\pt_{x_2}^2-\pt_{x_3}^2)^{-3}f\|_{H_x^8},
\nonumber
\end{eqnarray}
where $x=(x_1,x_2,x_3)$, $\langle x\rangle=\sqrt{1+|x|^2}$, 
and 
$P(-i\nabla)={{\mathcal F}}^{-1}P(\xi){{\mathcal F}}$ for 
$P=|\xi_1|^{-\delta}$ and $(-3\xi_1^2+\xi_2^2+\xi_3^2)^{-m}$, $m=2,3$. 
Let $\{V(t)\}_{t\in\rre}$ 
be a unitary group on $L^2$ generated by $-\pt_{x_1}\Delta$. 
Then we have the following. 

\begin{thm}\label{main}
Let $0<\delta<1$. 
Then 
for any $u_+\in X_{\delta}$,  
there exists a unique global solution 
$u\in C(\rre; H^1(\rre^2))$ to (\ref{ZK}) satisfying
\begin{eqnarray}
\|u(t)-V(t)u_+\|_{H_x^3}
\lesssim t^{-\alpha}\label{scattering}
\end{eqnarray}
for any $t\ge1$, where $\alpha>1/2$. 
Similar result holds for negative time direction. 
\end{thm}

\begin{rem} 
In \cite{S}, we proved the scattering result similar to 
Theorem \ref{main} in two dimensional case under the 
smallness assumption on the asymptotic states $u_+$. 
Note that in Theorem \ref{main}, we do not require 
smallness assumption on the asymptotic states 
thanks to good time decay of the free solution. 
\end{rem}

\begin{rem}\label{rem2}
The differential operator 
$3\pt_{x_1}^2-\pt_{x_2}^2-\pt_{x_3}^2$ 
in $X_{\delta}$ appears naturally in study of the linear/nonlinear 
scattering for (\ref{ZK}). 
For example, the solution to the linear equation for (\ref{ZK}) 
satisfies the following time decay estimate : 
\begin{eqnarray*}
\||\pt_{x_1}|^{\frac12}|3\pt_{x_1}^2-\pt_{x_2}^2-\pt_{x_3}^2|^{\frac12}
V(t)Pf\|_{L_x^{\infty}}
\lesssim t^{-\frac{3}{2}}\|f\|_{L_x^1},
\end{eqnarray*}
where $P$ is a suitable projection (see \cite[Theorem 3.2]{KPV}).
\end{rem}

\begin{rem} 
In the definition of the function space $X_{\delta}$, 
we required that $f\in{{\mathcal S}}$ for the simplicity of the argument. 
This requirement can be relaxed by the density argument for example. 
However, we do not discuss about it in this paper.
\end{rem}

\vskip2mm

We now give the outline of the proof of 
Theorem \ref{main}. As in \cite{S}, 
for given final data $u_+\in X_{\delta}$, 
we introduce a new unknown function
\begin{eqnarray*} 
w(t,x):=u(t,x)-u_1(t,x)-u_2(t,x),
\end{eqnarray*} 
where $u_1$ and $u_2$ are given by 
\begin{eqnarray} 
u_1(t,x)&=&[V(t)u_+](x),\label{v1}\\
u_2(t,x)&=&-
\pt_{x_1}\int_t^{\infty}
V(t-\tau)[u_1(\tau)^2]d\tau\label{v2}\\
&=&-
\pt_{x_1}\int_t^{\infty}
V(t-\tau)[(V(\tau)u_+)^2]d\tau.
\nonumber
\end{eqnarray} 
Let us derive the evolution equation for $w$. 
Let ${{\mathcal L}}=\pt_t+\pt_{x_1}\Delta$. 
Since ${{\mathcal L}}u_1=0$ and 
${{\mathcal L}}u_2=\pt_{x_1}(u_1^2)$, 
we have   
\begin{eqnarray*} 
{{\mathcal L}}w={{\mathcal L}}u-{{\mathcal L}}u_1-{{\mathcal L}}u_2
={{\mathcal L}}u-\pt_{x_1}(u_1^2).
\end{eqnarray*} 
If $u$ satisfies (\ref{ZK}), then 
\begin{eqnarray*} 
{{\mathcal L}}u
&=&\pt_{x_1}(u^2)\\
&=&\pt_{x_1}\{(w+u_1+u_2)^2\}\\
&=&\pt_{x_1}\{w^2+2(u_1+u_2)w+u_1^2+2u_1u_2+u_2^2\}\\
&=&\pt_{x_1}(w^2)+2\pt_{x_1}\{(u_1+u_2)w\}+\pt_{x_1}(u_1^2+2u_1u_2+u_2^2).
\end{eqnarray*} 
Hence we see that $w$ satisfies 
\begin{eqnarray} 
\ \ \ {{\mathcal L}}w
=\pt_{x_1}(w^2)+2\pt_{x_1}\left\{(u_1+u_2)w\right\}
+\pt_{x_1}(2u_1u_2+u_2^2).
\label{ZK11}
\end{eqnarray}
To show Theorem \ref{main}, we prove the existence  
of solution $w$ to (\ref{ZK11}) satisfying 
\begin{eqnarray*}
\sup_{t\in[T,\infty)}t^{\alpha}
\left(\|w(t)\|_{H_{x}^3}
+\||\pt_{x_1}|^{\frac{\nu}{2}}w(\tau)
\|_{L_{\tau}^{\frac{6}{3-2\nu}}(t,\infty;W_{x}^{2,\frac{2}{\nu}})}
\right)<\infty
\end{eqnarray*}
for suitable $\alpha>0$, $T>0$ and $0<\nu<1/2$. We first note that 
$w$ and $u_1$ can be easily estimated by 
the energy and the linear dispersive estimates (Strichartz estimates, 
see Lemma \ref{lemL} below). 
The main difficulty to prove Theorem \ref{main} 
lies on bilinear dispersive estimates for $u_2$. More precisely, we need 
$L^2$ estimate for the bilinear oscillatory integral : 
\begin{eqnarray} 
\ \ \ \ \ u_2&=&-\pt_{x_1}\int_t^{\infty}V(t-\tau)\left[(V(\tau)u_+)^2\right]d\tau
\label{u1}\\
&=&
-{{\mathcal F}}_{\xi\mapsto x}^{-1}
\left[\xi_1e^{it\xi_1|\xi|^2}
\int_t^{\infty}\!\!\!\int_{\rre^3}e^{-i\tau\phi(\xi,\eta)}
\widehat{u}_+(\xi-\eta)\widehat{u}_+(\eta)
d\eta d\tau\right],\nonumber
\end{eqnarray}  
where $\xi=(\xi_1,\xi_2,\xi_3)$, $\eta=(\eta_1,\eta_2,\eta_3)$ 
and $\phi(\xi,\eta)
=\xi_1|\xi|^2-(\xi_1-\eta_1)|\xi-\eta|^2-\eta_1|\eta|^2$. 
To derive time decay estimates for (\ref{u1}) in $L^2$, 
we employ so called 
space-time resonance method which is developed by 
Gustafson-Nakanishi-Tsai \cite{GNT2} and 
Germain-Masmoudi-Shatah \cite{GMS1,GMS2} etc. 
In this paper, 
we crucially use the ``null structure" of the nonlinear term 
which can be represented as the algebraic identity 
\begin{eqnarray}
\xi_1=
\left(\sum p(\xi)q(\eta)\right)\phi
+
\sum_{j=1}^3\left(\sum r(\xi)s(\eta)\right)\pt_{\eta_j}\phi
\label{algebra2}
\end{eqnarray}
with suitable polynomials $p,r$ and rational functions $q, s$, 
see Lemma \ref{algebra} below for detail. 
Combining (\ref{algebra2}) with  
integration by parts both in $\tau$ and $\eta$,  
we derive the time decay of (\ref{u1}) in $L^2$.  
Note that the null structure of the nonlinearity is employed 
in various contexts to study global dynamics of 
nonlinear PDEs 
since the pioneering work by Klainerman \cite{Kl}. 
We note that this approach was also used in \cite{S} 
to prove the existence of small global solution to (\ref{ZK}) 
with $d=2$ which scatters 
to the given free solution. Furthermore, in \cite{S}, 
the author transformed (\ref{gZK}) into 
the equation which is symmetric with respect to 
$x_1$ and $x_2$. Thanks to this transform, the problem 
became more transparent in two dimensional case. 
On the other hand, the similar transform  
is not known for (\ref{gZK}) with $d\neq2$. 
Therefore derivation of the key identity (\ref{algebra2}) for $d=3$ 
is more complicated compared to two dimensional case. 
%
Once we obtain $L^2$ estimate for $u_2$, 
we have an existence of solution to (\ref{ZK11}) 
by the compactness argument. 


We introduce several notations and function spaces 
which are used throughout this paper. 
For $f\in{{\mathcal S}}'(\rre^{3})$, $\hat{f}(\xi)$ 
denotes the Fourier transform of $f$. Let 
$\langle\xi\rangle=\sqrt{|\xi|^2+1}$. The differential operator 
$\langle\nabla\rangle^s=(1-\Delta)^{s/2}$ denotes 
the Bessel potential of order $-s$. 
For $1\le p,q\le\infty$, $L_{\tau}^p(t,\infty;L_x^q)$ is defined 
as follows:
\begin{eqnarray*}
L_{\tau}^p(t,\infty;L_x^q)&=&\{u\in{{\mathcal S}}'(\rre^{1+3});
\|u\|_{L_{\tau}^p(t,\infty;L_x^q)}<\infty\},\\
\|u\|_{L_{\tau}^p(t,\infty;L_x^q)}&=&
\left\|\|u(\tau)\|_{L_x^q}\right\|_{L_{\tau}^p(t,\infty)}.
\end{eqnarray*}
We will use the inhomogeneous Sobolev spaces
\begin{eqnarray*}
W^{s,q}=\lbrace f \in \mathcal S'(\rre^3); 
\|f\|_{W^{s,q}}=\| 
\langle \nabla\rangle^{s} f\|_{L^q}<\infty\rbrace,
\end{eqnarray*}
where $s\in\rre$ and $1\le q\le\infty$. 
We denote 
$H^s=W^{s,2}$. 
We denote $A \lesssim B$ if there exists a constant $C > 0$ 
such that $A \le C B$ holds and $ A \sim B $ if $ A \lesssim B \lesssim A $.

The outline of the paper is as follows. In Section 2, we 
give the decay and Strichartz estimates for the  
linearized equation of (\ref{ZK11}). In Section 3, we derive 
the key bilinear dispersive estimates. Finally,
in Section 4, we prove Theorem \ref{main}. 
In Appendix, we give the proof of Lemma \ref{algebra}. 

\vskip2mm

%
%
\section{Linear Dispersive Estimates}

In this section we derive the linear estimates associated with 
(\ref{ZK11}): 
\begin{equation} 
\left\{
\begin{array}{l}
\partial_tw+\pt_{x_1}\Delta w=0,\qquad\quad (t,x)\in
\mathbb{R}\times\mathbb{R}^{3}, \\
w(0,x)=f(x),\qquad\qquad\ x\in\mathbb{R}^{3}.
\end{array}
\right.\label{LZK} 
\end{equation}
Let us recall that $V(t)=e^{-t\pt_{x_1}\Delta}$  
is the unitary group on $L^2$ generated by 
$-\pt_{x_1}\Delta$. Then, the solution to (\ref{LZK}) 
can be written as $V(t)f$. 

We have the following decay and the Strichartz estimates for 
(\ref{LZK}).

\vskip2mm

\begin{lem}\label{lemL} 
(i) Let $0<a<1$ and $0\le b\le1$. Then for any $t>0$, we have
\begin{eqnarray}
\||\pt_{x_1}|^{ab}V(t)f\|_{L_{x}^q}
&\lesssim&t^{-b(1+\frac{a}{3})}
\|f\|_{L_{x}^{q'}},\label{linear1}
\end{eqnarray}
where $q=2/(1-b)$ and $q'$ is the H\"older conjugate exponent of $q$.

\vskip2mm
\noindent
(ii) Let $0<a<1$ and $0<b<(1+a/3)^{-1}$. 
Then, we have 
\begin{equation}
\left\||\pt_{x_1}|^{ab}
\int_{\tau}^{\infty}V(\tau-\tau')F(\tau')d\tau'
\right\|_{L_{\tau}^{p}(t,\infty;L_{x}^{q})}
\lesssim\|
F\|_{L_{\tau}^{p'}(t,\infty;L_{x}^{q'})},
\label{linear3}
\end{equation}
where $p=2/\{b(1+a/3)\}$ and $q=2/(1-b)$.
\end{lem}

\vskip2mm

\begin{proof}[Proof of Lemma \ref{lemL}.] 
See \cite[Lemma 3.3]{LS} for the proof of (\ref{linear1}), 
and \cite[Proposition 3.1]{LS}  for the proof of (\ref{linear3}). 
\end{proof}

\vskip2mm

%
%
\section{Bilinear Dispersive Estimates}

In this section we derive $L^2$ estimate for $u_2$ 
defined by (\ref{v2}) which is key to prove 
Theorem \ref{main}. We show the following.

\begin{prop}\label{xF}
Let $0<\delta<1$. Then for any $t>0$, we have 
\begin{eqnarray}
\ \ \ \ \ \left\|\pt_{x_1}\int_t^{\infty}V(t-\tau)
\left[(V(\tau)f)(V(\tau)g)\right]d\tau\right\|_{L_x^2}
\lesssim t^{-1-\frac{\delta}{3}}
\|f\|_{Y_{\delta}}\|g\|_{Y_{\delta}},\label{xFe}
\end{eqnarray}
where 
\begin{eqnarray*}
Y_{\delta}&=&
\{f\in{{\mathcal S}}(\rre^3) ; \|f\|_{Y_{\delta}}<\infty\},\\
\|f\|_{Y_{\delta}}&=&
\||\pt_{x_1}|^{-\delta}\langle x\rangle f\|_{W_x^{2,1}}
+\|(3\pt_{x_1}^2-\pt_{x_2}^2-\pt_{x_3}^2)^{-2}\langle x\rangle f\|_{H_x^3}
\\
& &+\|(3\pt_{x_1}^2-\pt_{x_2}^2-\pt_{x_3}^2)^{-3}f\|_{H_x^4}.
\end{eqnarray*}
\end{prop}

To prove Proposition \ref{xF}, we need to do some preparation. 
Simple calculation yields
\begin{eqnarray} 
\lefteqn{\pt_{x_1}\int_t^{\infty}V(t-\tau)
\left[(V(\tau)f)(V(\tau)g)\right]d\tau}\label{o1}\\
&=&
{{\mathcal F}}_{\xi\mapsto x}^{-1}
\left[\xi_1e^{it\xi_1|\xi|^2}
\int_t^{\infty}\!\!\!\int_{\rre^3}e^{-i\tau\phi(\xi,\eta)}
\widehat{f}(\xi-\eta)\widehat{g}(\eta)
d\eta d\tau\right](x),\nonumber
\end{eqnarray} 
where $\xi=(\xi_1,\xi_2,\xi_3), \eta=(\eta_1,\eta_2,\eta_3)$
and the resonant function $\phi$ is given by 
\begin{eqnarray} 
\phi(\xi,\eta)
&=&\xi_1|\xi|^2-(\xi_1-\eta_1)|\xi-\eta|^2-\eta_1|\eta|^2.
\label{r1}
\end{eqnarray} 
Therefore, to prove Proposition \ref{xF}, we need to estimate 
\begin{eqnarray} 
I(f,g):=\xi_1\int_t^{\infty}\!\!\!\int_{\rre^3}e^{-i\tau\phi(\xi,\eta)}
\widehat{f}(\xi-\eta)\widehat{g}(\eta)
d\eta d\tau. \label{Ifg}
\end{eqnarray} 
We evaluate (\ref{Ifg}) by using the space-time resonance method. 
To this end, we derive the following key algebraic identity. 

\begin{lem}\label{algebra} 
We have
\begin{eqnarray}
\xi_1=\psi_{\text{time}}(\xi,\eta)
+\psi_{\text{space}}(\xi,\eta),
\label{key2}
\end{eqnarray}
where $\psi_{\text{time}}$ and $\psi_{\text{space}}$ 
are given by 
\begin{eqnarray}
\psi_{\text{time}}(\xi,\eta)
&=&
\left(A_{0}(\eta)+\xi_1B_{0,1}(\eta)+\xi_2B_{0,2}(\eta)+\xi_3B_{0,3}(\eta)\right)\phi
\label{tr}\\
\ \ \ \psi_{\text{space}}(\xi,\eta)
&=&
\left(A_{1}(\eta)+\xi_1B_{1,1}(\eta)
+\xi_2B_{1,2}(\eta)+\xi_3B_{1,3}(\eta)\right)\pt_{\eta_1}\phi
\label{sr}\\
& &
+\left(A_2(\eta)+\xi_1B_{2,1}(\eta)
+\xi_2B_{2,2}(\eta)\right.\nonumber\\
& &\qquad\qquad\qquad\qquad\left.
+\xi_2\xi_3C_{2,1}(\eta)
+\xi_3^2C_{2,2}(\eta)\right)\pt_{\eta_2}\phi
\nonumber\\
& &
+\left(A_{3}(\eta)
+\xi_1B_{3,1}(\eta)
+\xi_3B_{3,2}(\eta)\right.\nonumber\\
& &\qquad\qquad\qquad\qquad\left.
+\xi_2\xi_3C_{3,1}(\eta)
+\xi_2^2C_{3,2}(\eta)\right)\pt_{\eta_3}\phi,
\nonumber
\end{eqnarray}
where $A_j$, $B_{j,k}$ and $C_{j,k}$ are rational functions 
in $\eta$ satisfying 
\begin{eqnarray*}
|A_0(\eta)|&\lesssim&(\eta_1^2+\eta_2^2+\eta_3^2)(3\eta_1^2-\eta_2^2-\eta_3^2)^{-2},\\
|\pt_{\eta_j}^iA_j(\eta)|
&\lesssim&\left\{
\begin{aligned}
&(\eta_1^2+\eta_2^2+\eta_3^2)^{\frac32}(3\eta_1^2-\eta_2^2-\eta_3^2)^{-2}
\qquad(i=0),\\
&(\eta_1^2+\eta_2^2+\eta_3^2)(3\eta_1^2-\eta_2^2-\eta_3^2)^{-2}\\
&+(\eta_1^2+\eta_2^2+\eta_3^2)^2|3\eta_1^2-\eta_2^2-\eta_3^2|^{-3}
\qquad(i=1)
\end{aligned}
\right.
\end{eqnarray*}
for $j=1,2,3$,
\begin{eqnarray*}
|B_{0,k}(\eta)|&\lesssim&(\eta_1^2+\eta_2^2+\eta_3^2)(3\eta_1^2-\eta_2^2-\eta_3^2)^{-2}
\end{eqnarray*}
for $k=1,2,3$,
\begin{eqnarray*}
|\pt_{\eta_j}^iB_{j,k}(\eta)|
&\lesssim&\left\{
\begin{aligned}
&(\eta_1^2+\eta_2^2+\eta_3^2)(3\eta_1^2-\eta_2^2-\eta_3^2)^{-2}
\quad(i=0),\\
&(\eta_1^2+\eta_2^2+\eta_3^2)^{\frac12}(3\eta_1^2-\eta_2^2-\eta_3^2)^{-2}\\
&+(\eta_1^2+\eta_2^2+\eta_3^2)^{\frac32}
|3\eta_1^2-\eta_2^2-\eta_3^2|^{-3}
\quad(i=1)
\end{aligned}
\right.
\end{eqnarray*}
for $(j,k)=(1,1), (1,2), (1,3), (2,1), (2,2), (3,1), (3,2)$, 
\begin{eqnarray*}
|\pt_{\eta_j}^iC_{j,k}(\eta)|
&\lesssim&\left\{
\begin{aligned}
&(\eta_1^2+\eta_2^2+\eta_3^2)^{\frac12}(3\eta_1^2-\eta_2^2-\eta_3^2)^{-2}
\quad(i=0),\\
&(3\eta_1^2-\eta_2^2-\eta_3^2)^{-2}\\
&+(\eta_1^2+\eta_2^2+\eta_3^2)|3\eta_1^2-\eta_2^2-\eta_3^2|^{-3}
\quad(i=1),
\end{aligned}
\right.
\end{eqnarray*}
for $(j,k)=(2,1), (2,2), (3,1), (3,2)$.

\end{lem}

We shall prove Lemma \ref{algebra} in Appendix.

\begin{proof}[Proof of Proposition \ref{xF}.] 
By using (\ref{key2}), we split $I(f,g)$ defined by (\ref{Ifg}) into 
the following two terms:
\begin{eqnarray}
I(f,g)
&=&
\int_t^{\infty}\!\!\!\int_{\rre^3}
\psi_{\text{time}}(\xi,\eta)e^{-i\tau\phi(\xi,\eta)}
\widehat{f}(\xi-\eta)
\widehat{g}(\eta)
d\eta d\tau\label{ST}\\
& &+\int_t^{\infty}\!\!\!\int_{\rre^3}
\psi_{\text{space}}(\xi,\eta)e^{-i\tau\phi(\xi,\eta)}
\widehat{f}(\xi-\eta)\widehat{g}(\eta)
d\eta d\tau\nonumber\\
&=:&I_{\text{time}}(f,g)+I_{\text{space}}(f,g),
\nonumber
\end{eqnarray}
where $\psi_{\text{time}}$ and $\psi_{\text{space}}$ 
are given by (\ref{tr}) and (\ref{sr}), respectively.

We first evaluate $I_{\text{time}}(f,g)$. 
We treat the second term of $\psi_{\text{time}}$, 
i.e., $\xi_1B_{0,1}(\eta)\phi$ only since 
the other terms can be treated in a similar way. 
Let
\begin{eqnarray*}
I_{\text{time},2}(f,g)
:=\xi_1\int_t^{\infty}\!\!\!\int_{\rre^3}
B_{0,1}(\eta)\phi(\xi,\eta)e^{-i\tau\phi(\xi,\eta)}
\widehat{f}(\xi-\eta)
\widehat{g}(\eta)
d\eta d\tau.
\end{eqnarray*}
Integrating in $\tau$, we have
\begin{eqnarray*}
I_{\text{time},2}(f,g)&=&
i\xi_1\int_t^{\infty}\!\!\!\int_{\rre^3}
\pt_{\tau}\left\{e^{-i\tau\phi(\xi,\eta)}\right\}
B_{0,1}(\eta)\widehat{f}(\xi-\eta)
\widehat{g}(\eta)
d\eta d\tau\\
&=&i\limsup_{T\to\infty}\xi_1\int_{\rre^3}
e^{-iT\phi(\xi,\eta)}
B_{0,1}(\eta)
\widehat{f}(\xi-\eta)
\widehat{g}(\eta)
d\eta\\
& &-i\xi_1\int_{\rre^3}
e^{-it\phi(\xi,\eta)}
B_{0,1}(\eta)
\widehat{f}(\xi-\eta)
\widehat{g}(\eta)
d\eta.
\end{eqnarray*}
Hence, by the Plancherel theorem 
and noting $\xi_1=(\xi_1-\eta_1)+\eta_1$, 
we obtain
\begin{eqnarray}
\lefteqn{\left\|I_{\text{time},2}(f,g)
\right\|_{L_{\xi}^2}}
\nonumber\\
&\lesssim&
\limsup_{T\to\infty}
\left\|\xi_1
\int_{\rre^3}
{{\mathcal F}}[V(T)f](\xi-\eta)
{{\mathcal F}}\left[
B_{0,1}(-i\nabla)V(T)g\right](\eta)
d\eta
\right\|_{L_{\xi}^2}
\nonumber\\
& &
+\left\|\xi_1
\int_{\rre^3}
{{\mathcal F}}[V(t)f](\xi-\eta)
{{\mathcal F}}\left[
B_{0,1}(-i\nabla)V(t)g\right](\eta)
d\eta
\right\|_{L_{\xi}^2}
\nonumber\\
&\lesssim&
\limsup_{T\to\infty}
\|\pt_{x_1}V(T)f\|_{L_{x}^{\infty}}
\left\|
B_{0,1}(-i\nabla)V(T)g
\right\|_{L_{x}^2}
\nonumber\\
& &+
\limsup_{T\to\infty}\|V(T)f\|_{L_{x}^{\infty}}
\left\|
\pt_{x_1}B_{0,1}(-i\nabla)V(T)g
\right\|_{L_{x}^2}
\nonumber\\
& &
+\|\pt_{x_1}V(t)f\|_{L_{x}^{\infty}}
\left\|
B_{0,1}(-i\nabla)V(t)g
\right\|_{L_{x}^2}
\nonumber\\
& &+
\|V(t)f\|_{L_{x}^{\infty}}
\left\|
\pt_{x_1}B_{0,1}(-i\nabla)V(t)g
\right\|_{L_{x}^2}.
\nonumber
\end{eqnarray}
By the decay estimate (Lemma \ref{lemL} (\ref{linear1})) 
and the inequality for $B_{0,1}$ in Lemma \ref{algebra}, 
we have
\begin{eqnarray*}
\left\|I_{\text{time},2}(f,g)
\right\|_{L_{\xi}^2}
\lesssim
t^{-1-\frac{\delta}{3}}\|f\|_{Y_{\delta}}\|g\|_{Y_{\delta}}.
\nonumber
\end{eqnarray*}
In a similar way we have
\begin{eqnarray}
\left\|I_{\text{time}}(f,g)\right\|_{L_{\xi}^2}
\lesssim t^{-1-\frac{\delta}{3}}\|f\|_{Y_{\delta}}\|g\|_{Y_{\delta}}.\label{LL}
\end{eqnarray}

Next we evaluate $I_{\text{space}}(f,g)$. 
We treat the second term of $\psi_{\text{space}}$, 
i.e., $\xi_1B_{1,1}(\eta)\pt_{\eta_1}\phi$ only since 
the other terms can be treated in a similar way. 
Let 
\begin{eqnarray*}
I_{\text{space,2}}(f,g)
:=
\xi_1\int_t^{\infty}\!\!\!\int_{\rre^3}
B_{1,1}(\eta)\pt_{\eta_1}\phi(\xi,\eta)e^{-i\tau\phi(\xi,\eta)}
\widehat{f}(\xi-\eta)
\widehat{g}(\eta)
d\eta d\tau.
\end{eqnarray*}
By an integration by parts in $\eta_1$, we have
\begin{eqnarray*}
I_{\text{space,2}}(f,g)&=&
i\xi_1\int_t^{\infty}\!\!\!\int_{\rre^3}
\tau^{-1}
\pt_{\eta_1}\left\{e^{-i\tau\phi(\xi,\eta)}\right\}
B_{1,1}(\eta)
\widehat{f}(\xi-\eta)
\widehat{g}(\eta)
d\eta d\tau\\
&=&
-i\xi_1\int_t^{\infty}\!\!\!\int_{\rre^3}
\tau^{-1}
e^{-i\tau\phi(\xi,\eta)}\pt_{\eta_1}B_{1,1}(\eta)
\widehat{f}(\xi-\eta)
\widehat{g}(\eta)
d\eta d\tau\\
& &
+i\xi_1\int_t^{\infty}\!\!\!\int_{\rre^3}
\tau^{-1}
e^{-i\tau\phi(\xi,\eta)}B_{1,1}(\eta)
\pt_{\eta_1}\widehat{f}(\xi-\eta)
\widehat{g}(\eta)
d\eta d\tau\\
& &
-i\xi_1\int_t^{\infty}\!\!\!\int_{\rre^3}
\tau^{-1}
e^{-i\tau\phi(\xi,\eta)}B_{1,1}(\eta)
\widehat{f}(\xi-\eta)
\pt_{\eta_1}\widehat{g}(\eta)
d\eta d\tau.
\end{eqnarray*}
Hence, by the Plancherel theorem 
and noting $\xi_1=(\xi_1-\eta_1)+\eta_1$, 
we obtain
\begin{eqnarray}
\lefteqn{\left\|
I_{\text{space},2}(f,g)
\right\|_{L_{\xi}^2}}
\nonumber\\
&\lesssim&
\biggl\|\xi_1
\int_t^{\infty}\!\!\!\int_{\rre^3}
\tau^{-1}{{\mathcal F}}[V(t)f](\xi-\eta)
{{\mathcal F}}\left[
(\pt_{\eta_1}B_{1,1})(-i\nabla)V(t)g\right](\eta)
d\eta d\tau
\biggl\|_{L_{\xi}^2}
\nonumber\\
& &+
\left\|\xi_1
\int_t^{\infty}\!\!\!\int_{\rre^3}
\tau^{-1}
{{\mathcal F}}[V(t)x_1f](\xi-\eta)
{{\mathcal F}}\left[
B_{1,1}(-i\nabla)V(t)g\right](\eta)
d\eta d\tau
\right\|_{L_{\xi}^2}
\nonumber\\
& &+
\left\|\xi_1
\int_t^{\infty}\!\!\!\int_{\rre^3}
\tau^{-1}
{{\mathcal F}}[V(t)f](\xi-\eta)
{{\mathcal F}}\left[
B_{1,1}(-i\nabla)V(t)x_1g\right](\eta)
d\eta d\tau
\right\|_{L_{\xi}^2}.
\nonumber
\end{eqnarray}
By the decay estimate (Lemma \ref{lemL} (\ref{linear1})) 
and the inequality for $B_{1,1}$ in Lemma \ref{algebra}, 
we have
\begin{eqnarray*}
\left\|
I_{\text{space},2}(f,g)
\right\|_{L_{\xi}^2}
\lesssim
t^{-1-\frac{\delta}{3}}\|f\|_{Y_{\delta}}\|g\|_{Y_{\delta}}.
\nonumber
\end{eqnarray*}
In a similar way we have
\begin{eqnarray}
\left\|I_{\text{space}}(f,g)\right\|_{L_{\xi}^2}
\lesssim t^{-1-\frac{\delta}{3}}\|f\|_{Y_{\delta}}\|g\|_{Y_{\delta}}.
\label{LL1}
\end{eqnarray}
Combining (\ref{o1}), (\ref{Ifg}), (\ref{ST}), (\ref{LL}) and (\ref{LL1}), 
we have 
\begin{eqnarray*}
\lefteqn{\left\|\pt_{x_1}\int_t^{\infty}V(t-\tau)
\left[(V(\tau)f)(V(\tau)g)\right]d\tau\right\|_{L_x^2}}
\qquad\qquad\qquad\\
&=&
\left\|I(f,g)
\right\|_{L_x^2}\\
&\lesssim&
\left\|I_{\text{time}}(f,g)\right\|_{L_{\xi}^2}
+\left\|I_{\text{space}}(f,g)\right\|_{L_{\xi}^2}\\
&\lesssim& t^{-1-\frac{\delta}{3}}\|f\|_{Y_{\delta}}\|g\|_{Y_{\delta}}.
\end{eqnarray*}
Hence we obtain (\ref{xFe}). 
This completes the proof of Proposition \ref{xF}. 
\end{proof}

\begin{lem}\label{cor:v2}
Let $0<\delta<1$ and 
let $u_1$ and $u_2$ be given by (\ref{v1}) and (\ref{v2}). 
Then for any $t>0$, we have
\begin{eqnarray}
\|u_1(t)\|_{W_x^{4,\infty}}
&\lesssim&t^{-1-\frac{\delta}{3}}\|u_+\|_{X_{\delta}},
\label{v11}\\
\|u_2(t)\|_{H_x^4}&\lesssim&
t^{-1-\frac{\delta}{3}}\|u_+\|_{X_{\delta}}^2,
\label{v22}
\end{eqnarray}
where the semi-norm $\|\cdot\|_{X_{\delta}}$ is given by (\ref{normX}). 
\end{lem}

\begin{proof}[Proof of Lemma \ref{cor:v2}] 
By Lemma \ref{lemL} (\ref{linear1}), we have
\begin{eqnarray*}
\|u_1(t)\|_{W_x^{4,\infty}}
=\|V(t)u_+\|_{W_x^{4,\infty}}
\lesssim t^{-1-\frac{\delta}{3}}
\||\pt_{x_1}|^{-\delta}u_+\|_{W_x^{4,1}},
\end{eqnarray*}
which yields (\ref{v11}).

To show (\ref{v22}), we note
\begin{eqnarray*}
\|u_2(t)\|_{H_x^4}
\sim
\|u_2(t)\|_{L_x^2}+\|\Delta^2 u_2(t)\|_{L_x^2}.
\end{eqnarray*}
We define
\begin{eqnarray*}
B(f,g):=\pt_{x_1}\int_t^{\infty}V(t-\tau)
\left[(V(\tau)f)(V(\tau)g)\right]d\tau.
\end{eqnarray*}
Since $\xi=(\xi-\eta)+\eta$, we see
\begin{eqnarray*} 
\nabla B(f,g)
=B(\nabla f,g)+B(f,\nabla g).
\end{eqnarray*} 
In a similar way, using the identity 
$|\xi|^2=|\xi-\eta|^2+2(\xi-\eta)\cdot\eta+|\eta|^2$, 
we have
\begin{eqnarray*}
\Delta B(f,g)=B(\Delta f,g)+
2\sum_{j=1}^3 B(\pt_{x_j} f,\pt_{x_j} g)+B(f,\Delta g).
\end{eqnarray*}
Therefore
\begin{eqnarray*}
\Delta^2 B(f,g)&=&
\Delta B(\Delta f,g)+2\sum_{j=1}^3\Delta B(\pt_{x_j} f,\pt_{x_j} g)
+\Delta B(f,\Delta g)\\
&=&
B(\Delta^2 f,g)+4\sum_{j=1}^3 B(\pt_{x_j} \Delta f,\pt_{x_j} g)\\
& &
+2B(\Delta f,\Delta g)+4\sum_{j=1}^3\sum_{k=1}^3B(\pt_{x_j}\pt_{x_k} f,\pt_{x_j}\pt_{x_k} g)\\
& &
+4\sum_{j=1}^3 B(\pt_{x_j} f,\pt_{x_j} \Delta g)+B(f,\Delta^2 g).
\end{eqnarray*}
Combining Proposition \ref{xF} (\ref{xFe}) with the above 
identities, we have (\ref{v22}).
\end{proof}

%
%
\section{Proof of Theorem \ref{main}}

In this section we complete the proof of Theorem \ref{main}.
To prove Theorem \ref{main}, we show the existence 
of solution $w$ to (\ref{ZK11}) with $w\to0$ in $H^3(\rre^3)$ 
as $t\to\infty$. 

Let 
\begin{eqnarray*}
N(w,u_1,u_2)
&:=&\pt_{x_1}(w^2)
+2\pt_{x_1}\left\{(u_1+u_2)w\right\}+\pt_{x_1}(2u_1u_2+u_2^2).
\end{eqnarray*}
Then (\ref{ZK11}) can be rewritten as 
\begin{eqnarray} 
\partial_tw+\pt_{x_1}\Delta w
=N(w,u_1,u_2).
\label{ZK12}
\end{eqnarray}
To show the existence of solution $w$ to (\ref{ZK11}) 
with $w\to0$ in $H^3(\rre^3)$, we consider 
the regularized equation associated with (\ref{ZK12}) : 
\begin{eqnarray}
\lefteqn{\partial_tw_{\lambda,\mu}
+\pt_{x_1}\Delta w_{\lambda,\mu}}\qquad\label{regular}\\
&=&(1+\lambda t)^{-5}\rho_{\mu}\ast 
N(\rho_{\mu}\ast w,\rho_{\mu}\ast u_1,\rho_{\mu}\ast u_2),
\nonumber
\end{eqnarray}
where $0<\lambda,\mu<1$, 
$\rho\in C_0^{\infty}(\rre^3)$ satisfies $\rho\ge0$ 
and $\int\rho(x)dx=1$, and $\rho_{\mu}(x)=
\mu^{-3}\rho(x/\mu)$. 

Thanks to the regularizing factor $\rho_{\mu}\ast$ 
and the time decaying factor $(1+\lambda t)^{-5}$, by using the contraction 
mapping principle, we easily see that for any 
$0<\lambda<1$ and $0<\mu<1$, there exists a 
$T_{\lambda,\mu}>0$ such that (\ref{regular}) has a 
unique solution $w_{\lambda,\mu}$ satisfying 
\begin{eqnarray*}
& &w_{\lambda,\mu}\in\bigcap_{j=1}^{\infty}
C^1([T_{\lambda,\mu},\infty),H_{x}^j),\\
& &\sup_{t\ge T_{\lambda,\mu}}(1+\lambda t)^4\sum_{3i+j\le 3}
\|\pt_t^i\nabla_{x}^jw_{\lambda,\mu}(t)\|_{L_{x}^2}<+\infty.
\end{eqnarray*}
Again using the regularizing and time decaying factors, the above 
solution $w_{\lambda,\mu}$ can be extend to $[0,\infty)$ 
without the smallness assumption on $u_+$.

We next derive an a priori estimates for $w_{\lambda,\mu}$ 
independent of $\lambda$ and $\mu$ under the assumption 
that $r:=\|u_+\|_{X_{\delta}}<\infty$,  where $\|\cdot\|_{X_{\delta}}$ 
is defined by (\ref{normX}). 
We abbreviate $w_{\lambda,\mu}$ to $w$. Let 
\begin{eqnarray*}
\|w\|_{Z_T}
:=\sup_{t\in[T,\infty)}t^{\alpha}
\left(\|w(t)\|_{H_{x}^3}
+\||\pt_{x_1}|^{\frac{\nu}{2}}w(\tau)
\|_{L_{\tau}^{\frac{6}{3-2\nu}}(t,\infty;W_{x}^{2,\frac{2}{\nu}})}
\right),
\end{eqnarray*}
where $\alpha>0$, $T>0$ and $0<\nu<1/2$ are fixed later. 

We first derive the estimates for $w$ in $H_{x}^3$. 

\begin{lem}\label{e1} 
Let $w$ be a solution to (\ref{regular}). 
Then we have
\begin{eqnarray}
\lefteqn{\sup_{t\in[T,\infty)}t^{\alpha}\|w(t)\|_{H_{x}^3}}
\label{5.12}\\
&\lesssim& 
(r^3+r^4)T^{\alpha-\frac{2}{3}\delta-1}
+(r+r^4)T^{-\frac{\delta}{3}}\|w\|_{Z_T}
+T^{-\alpha+\frac{\nu}{3}+\frac12}\|w\|_{Z_T}^2
\nonumber\\
& &
+(r+r^2)T^{-\alpha-\frac{\delta}{3}}\|w\|_{Z_T}^2
+T^{-2\alpha+1}\|w\|_{Z_T}^3
\nonumber\\
& &
+(r^3+r^4)T^{-11\alpha-\frac{2}{3}\delta-1}\|w\|_{Z_T}^{12}
+(r+r^2)T^{-12\alpha-\frac{\delta}{3}}\|w\|_{Z_T}^{13}, 
\nonumber
\end{eqnarray}
where the implicit constants are independent of 
$\lambda$ and $\mu$.
\end{lem}

\begin{proof}[Proof of Lemma \ref{e1}.] 
The proof is based on the energy method. 
Although we derive (\ref{5.12}) for smooth solution 
to (\ref{ZK12}), the proof for (\ref{ZK12}) below 
works for (\ref{regular}). 

Let $w$ be a smooth solution to (\ref{ZK12}). 
Taking the inner product in $L_x^2$ 
between (\ref{ZK12}) and $w$ and 
integrating by parts, we have
\begin{eqnarray}
\frac12\frac{d}{dt}\|w\|_{L_x^2}^2&=&
\int_{\rre^3}\pt_{x_1}(w^2)wdx
+2\int_{\rre^3}\pt_{x_1}\{(u_1+u_2)w\}wdx
\label{l1}\\
& &+\int_{\rre^3}\pt_{x_1}(2u_1u_2+u_2^2)wdx
\nonumber\\
&=&
-\int_{\rre^3}w^2\pt_{x_1}wdx
-2\int_{\rre^3}(u_1+u_2)w\pt_{x_1}wdx
\nonumber\\
& &
+\int_{\rre^3}\pt_{x_1}(2u_1u_2+u_2^2)wdx
\nonumber\\
&=&
\int_{\rre^3}\pt_{x_1}(u_1+u_2)w^2dx
+\int_{\rre^3}\pt_{x_1}(2u_1u_2+u_2^2)wdx
\nonumber\\
&\lesssim&
(\|u_1\|_{W_x^{1,\infty}}+\|u_2\|_{H_x^3})
\|w\|_{L_x^2}^2
\nonumber\\
& &+(\|u_1\|_{W_x^{1,\infty}}+\|u_2\|_{H_x^3})\|u_2\|_{H_x^3}
\|w\|_{L_x^2}.
\nonumber
\end{eqnarray}
Hence
\begin{eqnarray}
\frac{d}{dt}\|w\|_{L_x^2}^{14}
&=&7\|w\|_{L_x^2}^{12}\cdot\frac{d}{dt}\|w\|_{L_x^2}^2
\label{l2}\\
&\lesssim&
(\|u_1\|_{W_x^{1,\infty}}+\|u_2\|_{H_x^3})
\|w\|_{L_x^2}^{14}
\nonumber\\
& &+(\|u_1\|_{W_x^{1,\infty}}+\|u_2\|_{H_x^3})\|u_2\|_{H_x^3}
\|w\|_{L_x^2}^{13}.
\nonumber
\end{eqnarray}
Applying $\nabla\Delta$ to (\ref{ZK12}) and 
taking the inner product in $L_{x}^{2}$ between the resulting
equation and $\nabla\Delta w$, we obtain
\begin{eqnarray}
\frac12\frac{d}{dt}
\|\nabla\Delta w\|_{L_x^2}^2
&=&
\int_{\rre^3}\nabla\Delta\pt_{x_1}(w^2)\cdot
\nabla\Delta wdx\label{l3}\\
& &+2
\int_{\rre^3}\nabla\Delta\pt_{x_1}\{(u_1+u_2)w\}\cdot
\nabla\Delta wdx\nonumber\\
& &+\int_{\rre^3}\nabla\Delta\pt_{x_1}(2u_1u_2+u_2^2)\cdot
\nabla\Delta wdx\nonumber\\
&=:&I_1+I_2+I_3.\nonumber
\end{eqnarray}
For $I_1$, by integration by parts, we have
\begin{eqnarray*}
I_1&=&
-\int_{\rre^3}\nabla\Delta(w^2)\cdot\pt_{x_1}\nabla\Delta wdx\\
&=&
-\sum_{j=1}^3\int_{\rre^3}\pt_{x_j}\Delta(w^2)
\cdot\pt_{x_1}\pt_{x_j}\Delta wdx\\
&=&
-4\sum_{j=1}^3\int_{\rre^3}\nabla w\cdot\pt_{x_j}\nabla w
\pt_{x_1}\pt_{x_j}\Delta wdx
-2\sum_{j=1}^3\int_{\rre^3}\pt_{x_j}w\Delta w\pt_{x_1}\pt_{x_j}\Delta wdx\\
& &
-2\sum_{j=1}^3\int_{\rre^3}w\pt_{x_j}\Delta w\pt_{x_1}\pt_{x_j}\Delta wdx.
\end{eqnarray*}
By integration by parts again, we obtain
\begin{eqnarray}
I_1&=&
4\sum_{j=1}^3\int_{\rre^3}|\pt_{x_j}\nabla w|^2
\pt_{x_1}\Delta wdx
+4\int_{\rre^3}\nabla w\cdot\nabla\Delta w
\pt_{x_1}\Delta wdx
\label{l4}\\
& &
+2\int_{\rre^3}\nabla w\cdot\nabla\Delta w\pt_{x_1}\Delta wdx
+\int_{\rre^3}\pt_{x_1}w|\nabla\Delta w|^2dx
\nonumber\\
&=&
6\int_{\rre^3}\nabla w\cdot\nabla\Delta w\pt_{x_1}\Delta wdx
+\int_{\rre^3}\pt_{x_1}w|\nabla\Delta w|^2dx
+R_1,
\nonumber
\end{eqnarray}
where 
\begin{eqnarray}
R_1
&=&4\sum_{j=1}^3\int_{\rre^3}|\pt_{x_j}\nabla w|^2
\pt_{x_1}\Delta wdx
\label{R1}\\
&=&-8\sum_{j,k=1}^3\int_{\rre^3}\pt_{x_j}\nabla w
\cdot \pt_{x_j}\pt_{x_k}\nabla w
\pt_{x_1}\pt_{x_k}wdx
\nonumber\\
&\lesssim&\|\pt_{x_1}w\|_{W_x^{1,\frac{2}{\nu}}}\|w\|_{H_x^3}^2.
\nonumber
\end{eqnarray}
In a similar way, we see
\begin{eqnarray}
|I_2|&\lesssim&(\|u_1\|_{W_x^{4,\infty}}+\|u_2\|_{H_x^4})\|w\|_{H_x^3}^2,
\label{R2}\\
|I_3|&\lesssim&(\|u_1\|_{W_x^{4,\infty}}+\|u_2\|_{H_x^4})\|u_2\|_{H_x^4}\|w\|_{H_x^3}.
\label{R3}
\end{eqnarray}
On the other hand
\begin{eqnarray*}
\frac{d}{dt}\int_{\rre^3}w(\Delta w)^2dx&=&
\int_{\rre^3}\pt_tw(\Delta w)^2dx
+2\int_{\rre^3}w\Delta w\pt_t\Delta wdx\\
&=&-2\int_{\rre^3}w\Delta w\pt_{x_1}\Delta^2 wdx+R_2,
\end{eqnarray*}
where
\begin{eqnarray}
R_2
&=&
-\int_{\rre^3}\pt_{x_1}\Delta w(\Delta w)^2dx
+\int_{\rre^3}N(w,u_1,u_2)(\Delta w)^2dx
\label{R4}\\
& &
+2\int_{\rre^3}w\Delta w\Delta N(w,u_1,u_2)dx
\nonumber\\
&\lesssim&\|w\|_{H_x^3}^4
+(\|u_1\|_{W_x^{3,\infty}}+\|u_2\|_{H_x^3})\|w\|_{H_x^3}^3
\nonumber\\
& &+(\|u_1\|_{W_x^{3,\infty}}+\|u_2\|_{H_x^3})\|u_2\|_{H_x^3}\|w\|_{H_x^3}^2.
\nonumber
\end{eqnarray}
By integration by parts, we obtain
\begin{eqnarray}
\lefteqn{\frac{d}{dt}\int_{\rre^3}w(\Delta w)^2dx}
\label{l5}\\
&=&
2\int_{\rre^3}\nabla w\Delta w\cdot \pt_{x_1}\nabla\Delta wdx
+2\int_{\rre^3}w\nabla\Delta w\cdot \pt_{x_1}\nabla\Delta wdx
\nonumber\\
& &+R_2
\nonumber\\
&=&
-2\int_{\rre^3}\nabla w\cdot\nabla\Delta w\pt_{x_1}\Delta wdx
-\int_{\rre^3}\pt_{x_1}w|\nabla\Delta w|^2dx
+R_2.
\nonumber
\end{eqnarray}
Hence, from (\ref{l3}), (\ref{l4}) and (\ref{l5}), 
we see that for any $M>0$, 
\begin{eqnarray*}
\lefteqn{
\frac{d}{dt}\left(\|w(t)\|_{H_{x}^3}^2
+6\int_{\rre^3}w(\Delta w)^2dx+M\|w\|_{L_x^2}^{14}\right)}\qquad\\
&=&-4\int_{\rre^3}\pt_{x_1}w|\nabla\Delta w|^2dx
+\frac{d}{dt}\|w\|_{L_x^2}^2+M\frac{d}{dt}\|w\|_{L_x^2}^{14}\\
& &+2I_2+2I_3+2R_1+R_2.
\end{eqnarray*}
By (\ref{l1}), (\ref{l2}), (\ref{R1}), (\ref{R2}), (\ref{R3}) and (\ref{R4}), we have
\begin{eqnarray*}
\lefteqn{
\|w(t)\|_{H_{x}^3}^2
+6\int_{\rre^3}w(\Delta w)^2dx+M\|w\|_{L_x^2}^{14}}\qquad\\
&\lesssim&
\int_t^{+\infty}
(\|u_1(\tau)\|_{W_{x}^{4,\infty}}
+\|u_2(\tau)\|_{H_x^4})\|u_2(\tau)\|_{H_x^4}\|w(\tau)\|_{H_{x}^3}
d\tau\\
& &+\int_t^{+\infty}
(\|u_1(\tau)\|_{W_{x}^{4,\infty}}
+\|u_2(\tau)\|_{H_x^4})
\|w(\tau)\|_{H_{x}^3}^2d\tau\\
& &+\int_t^{+\infty}
(\|u_1(\tau)\|_{W_{x}^{3,\infty}}
+\|u_2(\tau)\|_{H_x^3})\|u_2(\tau)\|_{H_x^3}
\|w(\tau)\|_{H_{x}^3}^2d\tau\\
& &+\int_t^{+\infty}
(\|\pt_{x_1}w(\tau)\|_{L_{x}^{\infty}}
+\|\pt_{x_1}w(\tau)\|_{W_{x}^{1,\frac{2}{\nu}}})
\|w(\tau)\|_{H_{x}^3}^2d\tau\\
& &+\int_t^{+\infty}
(\|u_1(\tau)\|_{W_x^{3,\infty}}+\|u_2(\tau)\|_{H_x^3})\|w(\tau)\|_{H_x^3}^3
d\tau\\
& &+\int_t^{+\infty}
\|w(\tau)\|_{H_x^3}^4
d\tau\\
& &+\int_t^{+\infty}(\|u_1(\tau)\|_{W_x^{1,\infty}}+\|u_2(\tau)\|_{H_x^3})
\|u_2(\tau)\|_{H_x^3}\|w(\tau)\|_{L_x^2}^{13}d\tau\\
& &+\int_t^{+\infty}(\|u_1(\tau)\|_{W_x^{1,\infty}}+\|u_2(\tau)\|_{H_x^3})
\|w(\tau)\|_{L_x^2}^{14}d\tau.
\nonumber
\end{eqnarray*} 
By the Sobolev embedding, we find
\begin{eqnarray*}
\|\pt_{x_1}w\|_{L_{x}^{\infty}}
\lesssim
\|\pt_{x_1}w\|_{W_{x}^{2\nu,\frac{2}{\nu}}}
\lesssim
\||\pt_{x_1}|^{\frac{\nu}{2}}w
\|_{W_{x}^{1+\frac32\nu,\frac{2}{\nu}}}
\lesssim
\||\pt_{x_1}|^{\frac{\nu}{2}}w
\|_{W_{x}^{2,\frac{2}{\nu}}}.
\end{eqnarray*}
Hence the H\"older inequality and Lemma \ref{cor:v2} yield 
\begin{eqnarray}
\lefteqn{\|w(t)\|_{H_{x}^3}^2
+6\int_{\rre^3}w(\Delta w)^2dx+M\|w\|_{L_x^2}^{14}}
\nonumber\\
&\lesssim&
(r^3+r^4)\int_t^{+\infty}\tau^{-2-\frac{2}{3}\delta}
\|w(\tau)\|_{H_{x}^3}d\tau
+
(r+r^4)\int_t^{+\infty}\tau^{-1-\frac{\delta}{3}}
\|w(\tau)\|_{H_{x}^3}^2d\tau
\nonumber\\
& &+\|w(\tau)\|_{L_{\tau}^{\frac{6}{3+2\nu}}(t,\infty;H_{x}^3)}^2
\||\pt_{x_1}|^{\frac{\nu}{2}}
w(\tau)\|_{L_{\tau}^{\frac{6}{3-2\nu}}(t,\infty;W_{x}^{2,\frac{2}{\nu}})}
\nonumber\\
& &+(r+r^2)
\int_t^{+\infty}\tau^{-1-\frac{\delta}{3}}\|w(\tau)\|_{H_{x}^3}^3d\tau
+\int_t^{+\infty}\|w(\tau)\|_{H_{x}^3}^4d\tau
\nonumber\\
& &+(r^3+r^4)
\int_t^{+\infty}\tau^{-2-\frac{2}{3}\delta}\|w(\tau)\|_{H_{x}^3}^{13}d\tau
+(r+r^2)\int_t^{+\infty}\tau^{-1-\frac{\delta}{3}}\|w(\tau)\|_{H_{x}^3}^{14}d\tau.
\nonumber
\end{eqnarray}
Therefore,
\begin{eqnarray}
\label{y1}\\
\lefteqn{\|w(t)\|_{H_{x}^3}^2
+6\int_{\rre^3}w(\Delta w)^2dx+M\|w\|_{L_x^2}^{14}}
\nonumber\\
&\lesssim&
(r^3+r^4)t^{-\alpha-\frac{2}{3}\delta-1}
\left(\sup_{t\in[T,\infty)}t^{\alpha}\|w(t)\|_{H_{x}^3}\right)
\nonumber\\
& &
+(r+r^4)t^{-2\alpha-\frac{\delta}{3}}
\left(\sup_{t\in[T,\infty)}t^{\alpha}\|w(t)\|_{H_{x}^3}\right)^2
\nonumber\\
& &+t^{-3\alpha+\frac{\nu}{3}+\frac12}
\left(\sup_{t\in[T,\infty)}t^{\alpha}\|w(t)\|_{H_{x}^3}\right)^2
\nonumber\\
& &\qquad\qquad\qquad\quad\quad\times
\left(\sup_{t\in[T,\infty)}t^{\alpha}
\||\pt_{x_1}|^{\frac{\nu}{2}}
w(\tau)\|_{L_{\tau}^{\frac{6}{3-2\nu}}
(t,\infty;W_{x}^{2,\frac{2}{\nu}})}\right)
\nonumber\\
& &
+(r+r^2)t^{-3\alpha-\frac{\delta}{3}}
\left(\sup_{t\in[T,\infty)}t^{\alpha}\|w(t)\|_{H_{x}^3}\right)^3
+t^{-4\alpha+1}\left(\sup_{t\in[T,\infty)}t^{\alpha}\|w(t)\|_{H_{x}^3}\right)^4
\nonumber\\
& &
+(r^3+r^4)t^{-13\alpha-\frac{2}{3}\delta-1}
\left(\sup_{t\in[T,\infty)}t^{\alpha}\|w(t)\|_{H_{x}^3}\right)^{13}
\nonumber\\
& &
+(r+r^2)t^{-14\alpha-\frac{\delta}{3}}
\left(\sup_{t\in[T,\infty)}t^{\alpha}\|w(t)\|_{H_{x}^3}\right)^{14}
\nonumber
\end{eqnarray}
for any $t\in[T,\infty)$, 
where the implicit constants are independent of $\lambda$ and $\mu$. 
By the Gagliardo-Nirenberg inequality 
\begin{eqnarray*}
\|\Delta w\|_{L_x^4}
\lesssim\|w\|_{L_x^2}^{\frac{1}{12}}\|w\|_{H_x^3}^{\frac{11}{12}},
\end{eqnarray*}
we have
\begin{eqnarray*}
6\int_{\rre^3}w(\Delta w)^2dx
&\le&6\|w\|_{L_x^2}\|\Delta w\|_{L_x^4}^2\\
&\le&C\|w\|_{L_x^2}^{\frac{7}{6}}\|w\|_{H_x^3}^{\frac{11}{6}}\
\le C'\|w\|_{L_x^2}^{14}+\frac12\|w\|_{H_x^3}^2.
\end{eqnarray*}
Thus, if $M>0$ is sufficiently large, then 
\begin{eqnarray}
\|w(t)\|_{H_{x}^3}^2
\sim
\|w(t)\|_{H_{x}^3}^2
+6\int_{\rre^3}w(\Delta w)^2dx+M\|w\|_{L_x^2}^{14}.
\label{y2}
\end{eqnarray}
Combining (\ref{y1}) with (\ref{y2}), we have (\ref{5.12}). 
\end{proof}

Next we derive the estimates for 
$|\pt_{x_1}|^{\frac{\nu}{2}}w$ 
in $L_{\tau}^{\frac{6}{3-2\nu}}(t,\infty;W_{x}^{2,\frac{2}{\nu}})$. 

\begin{lem}\label{e2}
Let $w$ be a solution to (\ref{regular}). 
Then we have
\begin{eqnarray}
\lefteqn{\sup_{t\in[T,\infty)}t^{\alpha}\||\pt_{x_1}|^{\frac{\nu}{2}}w(\tau)
\|_{L_{\tau}^{\frac{6}{3-2\nu}}(t,\infty;W_{x}^{2,\frac{\nu}{2}})}}
\qquad\qquad\label{5.4}\\
&\lesssim&
(r^3+r^4)T^{\alpha-\frac{2}{3}\delta-1}
+(r+r^2)T^{-\frac{\delta}{3}}\|w\|_{Z_T}
\nonumber\\
& &+T^{-\alpha+\frac{\nu}{3}+\frac12}\|w\|_{Z_T}^2,
\nonumber
\end{eqnarray}
where the implicit constants are independent of 
$\lambda$ and $\mu$.
\end{lem}

\begin{proof}[Proof of Lemma \ref{e2}.] 
Since $w$ satisfies 
\begin{eqnarray*}
w(t)&=&-
(1+\lambda t)^{-5}\rho_{\mu}\ast\pt_{x_1}
\int_t^{\infty}V(t-\tau)\left[w(\tau)^2\right]d\tau\\
& &-2
(1+\lambda t)^{-5}\rho_{\mu}\ast\pt_{x_1}\int_t^{\infty}V(t-\tau)
\left[(u_1(\tau)+u_2(\tau))w(\tau)\right]d\tau\\
& &-
(1+\lambda t)^{-5}\rho_{\mu}\ast\pt_{x_1}\int_t^{\infty}V(t-\tau)
\left[2u_1(\tau)u_2(\tau)+u_2(\tau)^2\right]d\tau, 
\end{eqnarray*}
applying the Strichartz estimates (Lemma \ref{lemL} (\ref{linear3})), 
we have
\begin{eqnarray}
\label{cF}\\
\lefteqn{\||\pt_{x_1}|^{\frac{\nu}{2}}w(\tau)
\|_{L_{\tau}^{\frac{6}{3-2\nu}}(t,\infty;W_{x}^{2,\frac{2}{\nu}})}}
\nonumber\\
&\lesssim&\|
|\pt_{x_1}|^{1-\frac{\nu}{2}}w(\tau)^2
\|_{L_{\tau}^{\frac{6}{3+2\nu}}(t,\infty;W_{x}^{2,\frac{2}{2-\nu}})}
+\left\|(u_1(\tau)+u_2(\tau))w(\tau)\right\|_{L_{\tau}^1(t,\infty;H_{x}^3)}
\nonumber\\
& &
+\|2u_1(\tau)u_2(\tau)+u_2(\tau)^2\|_{L_{\tau}^1(t,\infty;H_{x}^{3})}.
\nonumber
\end{eqnarray}
By the H\"{o}lder inequality, 
\begin{eqnarray}
\lefteqn{\||\pt_{x_1}|^{1-\frac{\nu}{2}}w(\tau)^2
\|_{L_{\tau}^{\frac{6}{3+2\nu}}(t,\infty;W_{x}^{2,\frac{2}{2-\nu}})}}
\qquad\qquad\label{t1}\\
&\lesssim&\|w(\tau)^2
\|_{L_{\tau}^{\frac{6}{3+2\nu}}(t,\infty;W_{x}^{3,\frac{2}{2-\nu}})}
\nonumber\\
&\lesssim&\|\|w(\tau)\|_{H_{x}^3}^2\|_{L_{\tau}^{\frac{6}{3+2\nu}}(t,\infty)}
\nonumber\\
&\lesssim&
\left(\sup_{t\in[T,\infty)}t^{\alpha}\|w(t)\|_{H_{x}^3}\right)^2
\|\tau^{-2\alpha}\|_{L_{\tau}^{\frac{6}{3+2\nu}}(t,\infty)}
\nonumber\\
&\lesssim&t^{-2\alpha+\frac{\nu}{3}+\frac12}\|w\|_{Z_T}^2.
\nonumber
\end{eqnarray}
By the H\"{o}lder inequality and Lemma \ref{cor:v2}, 
\begin{eqnarray}
\lefteqn{\|(u_1(\tau)+u_2(\tau))w(\tau)\|_{L^1(t,\infty;H_{x}^3)}}
\qquad\qquad
\label{t2}\\
&\lesssim&
\left\|(\|u_1(\tau)\|_{W_{x}^{3,\infty}}
+\|u_2(\tau)\|_{H_x^3})
\|w(\tau)\|_{H_{x}^3}\right\|_{L_{\tau}^1(t,\infty)}
\nonumber\\
&\lesssim&(r+r^2)
\left(\sup_{t\in[T,\infty)}t^{\alpha}\|w(t)\|_{H_{x}^3}\right)
\|\tau^{-1-\alpha-\frac{\delta}{3}}\|_{L_{\tau}^1(t,\infty)}
\nonumber\\
&\lesssim&(r+r^2)t^{-\alpha-\frac{\delta}{3}}
\|w\|_{Z_T},\nonumber\\
\lefteqn{\|2u_1(\tau)u_2(\tau)+u_2(\tau)^2\|_{L^1(t,\infty;H_{x}^3)}}
\qquad\qquad
\label{t3}\\
&\lesssim&
\left\|\|u_1(\tau)\|_{W_{x}^{3,\infty}}
\|u_2(\tau)\|_{H_x^3}
+\|u_2(\tau)\|_{H_x^3}^2\right\|_{L_{\tau}^1(t,\infty)}
\nonumber\\
&\lesssim&(r^3+r^4)\|\tau^{-2-\frac{2}{3}\delta}\|_{L_{\tau}^1(t,\infty)}
\nonumber\\
&\lesssim&(r^3+r^4)t^{-1-\frac{2}{3}\delta}.
\nonumber
\end{eqnarray} 
Substituting (\ref{t1}), (\ref{t2}), (\ref{t3}) into (\ref{cF}), we have 
(\ref{5.4}), where the implicit constants are independent of 
$\lambda$ and $\mu$.
\end{proof}

\begin{proof}[Proof of Theorem \ref{main}.] 
By Lemmas \ref{e1} 
and 
\ref{e2}, 
we have 
\begin{eqnarray*}
\|w\|_{Z_T}&\lesssim&
(r^3+r^4)T^{\alpha-\frac{2}{3}\delta-1}
+(r+r^4)T^{-\frac{\delta}{3}}\|w\|_{Z_T}
+T^{-\alpha+\frac{\nu}{3}+\frac12}\|w\|_{Z_T}^2\\
& &
+(r+r^2)T^{-\alpha-\frac{\delta}{3}}\|w\|_{Z_T}^2
+T^{-2\alpha+1}\|w\|_{Z_T}^3\\
& &
+(r^3+r^4)T^{-11\alpha-\frac{2}{3}\delta-1}\|w\|_{Z_T}^{12}
+(r+r^2)T^{-12\alpha-\frac{\delta}{3}}\|w\|_{Z_T}^{13}.
\end{eqnarray*}
We now choose $\alpha,\nu>0$ so that 
$1/2<\alpha<1$ and $\nu/3+1/2<\alpha$. Then, 
we see that there exists $T>0$ 
which depends on $r$ and is independent of $\lambda$ and $\mu$ 
such that for any $0<\lambda<1$ and $0<\mu<1$,
\begin{eqnarray}
\|w\|_{Z_T}
\le 2r.\label{5.5}
\end{eqnarray}
Combining a priori estimate (\ref{5.5}) with 
the standard compactness argument 
(see \cite[Section 3]{OT} for instance), we find that 
there exists a unique solution 
$u\in C([T,\infty) ; H_x^1(\rre))$ to (\ref{ZK}) 
which satisfies $\|w\|_{Z_T}
=\|u-u_1-u_2\|_{Z_T}\le2r$.  
By conservations of the mass (\ref{mass}) and 
the energy (\ref{energy}), we see $u\in C(\rre ; H^1(\rre))$. 
Furthermore, from the above inequality 
and Lemma \ref{cor:v2}, we see
\begin{eqnarray*}
\|u(t)-V(t)u_+\|_{H_x^3}&\lesssim&
\|w(t)\|_{H_x^3}+\|u_2(t)\|_{H_x^3}\\
&\lesssim&
rt^{-\alpha}+r^2 t^{-1-\frac{\delta}{3}}\\
&\lesssim&(r+r^2) t^{-\alpha}
\end{eqnarray*}
for any $t\ge1$. 
This completes 
the proof of Theorem \ref{main}. \end{proof}


%
%

\appendix

\section{Proof of Lemma \ref{algebra}.}

In this appendix, we prove Lemma \ref{algebra}.

\begin{proof}[Proof of Lemma \ref{algebra}.] By (\ref{r1}), we see 
\begin{eqnarray} 
\pt_{\eta_1}\phi(\xi,\eta)
&=&
3\xi_1^2+\xi_2^2+\xi_3^2
-6\xi_1\eta_1-2\xi_2\eta_2-2\xi_3\eta_3,
\label{r2}\\
\pt_{\eta_2}\phi(\xi,\eta)&=&2\xi_1\xi_2-2\xi_1\eta_2-2\eta_1\xi_2,
\label{r3}\\
\pt_{\eta_3}\phi(\xi,\eta)&=&2\xi_1\xi_3-2\xi_1\eta_3-2\eta_1\xi_3.
\label{r4}
\end{eqnarray} 
Hence, we have
\begin{eqnarray*} 
\phi(\xi,\eta)-\eta\cdot\nabla_{\eta}\phi(\xi,\eta)
=(3\eta_1^2+\eta_2^2+\eta_3^2)\xi_1+2\eta_1(\eta_2\xi_2+\eta_3\xi_3).
\end{eqnarray*} 
Therefore, 
\begin{eqnarray} 
\eta_2\xi_2+\eta_3\xi_3
=
-\frac{3\eta_1^2+\eta_2^2+\eta_3^2}{2\eta_1}\xi_1
+N_1,
\label{r7}
\end{eqnarray} 
where $N_1=N_1(\xi,\eta)$ is given by 
\begin{eqnarray} 
N_1
=\frac{1}{2\eta_1}
\left(\phi(\xi,\eta)-\eta\cdot\nabla_{\eta}\phi(\xi,\eta)\right).
\label{n1}
\end{eqnarray} 
On the other hand, by 
$(\xi_3-\eta_3)\times$(\ref{r3})$-(\xi_2-\eta_2)\times$(\ref{r4}),
\begin{eqnarray*} 
2\eta_1(\eta_3\xi_2-\eta_2\xi_3)
&=&(\xi_3-\eta_3)\pt_{\eta_2}\phi-(\xi_2-\eta_2)\pt_{\eta_3}\phi.
\end{eqnarray*} 
Therefore,
\begin{eqnarray} 
\eta_3\xi_2-\eta_2\xi_3
&=&
\frac{1}{2\eta_1}
\left\{(\xi_3-\eta_3)\pt_{\eta_2}\phi-(\xi_2-\eta_2)\pt_{\eta_3}\phi\right\}.
\label{r8}
\end{eqnarray} 
By (\ref{r7}) and (\ref{r8}),
\begin{eqnarray*}
(\eta_2^2+\eta_3^2)(\xi_2^2+\xi_3^2)
&=&(\eta_2\xi_2+\eta_3\xi_3)^2+(\eta_3\xi_2-\eta_2\xi_3)^2\\
&=&
\frac{(3\eta_1^2+\eta_2^2+\eta_3^2)^2}{4\eta_1^2}\xi_1^2
+(\eta_2^2+\eta_3^2)N_2(\xi,\eta),
\end{eqnarray*}
where $N_2=N_2(\xi,\eta)$ is given by 
\begin{eqnarray}
\label{n2}\\
N_2
&=&-\frac{3\eta_1^2+\eta_2^2+\eta_3^2}{2\eta_1^2(\eta_2^2+\eta_3^2)}
\xi_1
\left(\phi-\eta\cdot\nabla_{\eta}\phi\right)
+
\frac{1}{4\eta_1^2(\eta_2^2+\eta_3^2)}
\left(\phi-\eta\cdot\nabla_{\eta}\phi\right)^2
\nonumber\\
& &
+\frac{1}{4\eta_1^2(\eta_2^2+\eta_3^2)}
\left\{(\xi_3-\eta_3)\pt_{\eta_2}\phi-(\xi_2-\eta_2)\pt_{\eta_3}\phi\right\}^2
\nonumber\\
&=&
\frac{1}{4\eta_1^2(\eta_2^2+\eta_3^2)}
\{-(3\eta_1^2+\eta_2^2+\eta_3^2)\xi_1
+2\eta_1\eta_2\xi_2+2\eta_1\eta_3\xi_3\}\phi
\nonumber\\
& &
+\frac{1}{4\eta_1^2(\eta_2^2+\eta_3^2)}
\{(3\eta_1^2+\eta_2^2+\eta_3^2)\xi_1
-2\eta_1\eta_2\xi_2-2\eta_1\eta_3\xi_3\}
\eta_1\pt_{\eta_1}\phi
\nonumber\\
& &
+\frac{1}{4\eta_1^2(\eta_2^2+\eta_3^2)}
\{2\eta_1\eta_3\xi_2\xi_3
-2\eta_1\eta_2\xi_3^2
+(3\eta_1^2+\eta_2^2+\eta_3^2)\eta_2\xi_1
\nonumber\\
& &
\qquad\qquad\qquad\ 
-2\eta_1(\eta_2^2+\eta_3^2)\xi_2
\}\pt_{\eta_2}\phi
\nonumber\\
& &
+\frac{1}{4\eta_1^2(\eta_2^2+\eta_3^2)}
\{-2\eta_1\eta_3\xi_2^2
+2\eta_1\eta_2\xi_2\xi_3
+(3\eta_1^2+\eta_2^2+\eta_3^2)\eta_3\xi_1
\nonumber\\
& &
\qquad\qquad\qquad\ 
-2\eta_1(\eta_2^2+\eta_3^2)\xi_3
\}\pt_{\eta_3}\phi.
\nonumber
\end{eqnarray}
Hence
\begin{eqnarray}
\xi_2^2+\xi_3^2
=
\frac{(3\eta_1^2+\eta_2^2+\eta_3^2)^2}{4\eta_1^2(\eta_2^2+\eta_3^2)}\xi_1^2
+N_2(\xi,\eta).\label{r9}
\end{eqnarray}
Substituting (\ref{r7}) and 
(\ref{r9}) into (\ref{r2}), we have
\begin{eqnarray*} 
\lefteqn{\pt_{\eta_1}\phi(\xi,\eta)}\\
&=&
(3\xi_1^2-6\xi_1\eta_1)+(\xi_2^2+\xi_3^2)
-2(\eta_2\xi_2+\eta_3\xi_3)
\nonumber\\
&=&
(3\xi_1^2-6\xi_1\eta_1)
+
\frac{(3\eta_1^2+\eta_2^2+\eta_3^2)^2}{4\eta_1^2(\eta_2^2+\eta_3^2)}\xi_1^2
+
\frac{3\eta_1^2+\eta_2^2+\eta_3^2}{\eta_1}\xi_1-2N_1+N_2
\nonumber\\
&=&
\frac{9\eta_1^4+18\eta_1^2(\eta_2^2+\eta_3^2)+(\eta_2^2+\eta_3^2)^2}{4\eta_1^2(\eta_2^2+\eta_3^2)}\xi_1^2
-\frac{3\eta_1^2-\eta_2^2-\eta_3^2}{\eta_1}\xi_1
-2N_1+N_2.\nonumber
\end{eqnarray*} 
Therefore
\begin{eqnarray} 
\lefteqn{
\frac{9\eta_1^4+18\eta_1^2(\eta_2^2+\eta_3^2)+(\eta_2^2+\eta_3^2)^2}{4\eta_1^2(\eta_2^2+\eta_3^2)}\xi_1^2
-\frac{3\eta_1^2-\eta_2^2-\eta_3^2}{\eta_1}\xi_1}
\qquad\qquad\qquad\qquad\qquad\qquad\qquad\qquad
\label{r12}\\
&=&\pt_{\eta_1}\phi+2N_1-N_2.\nonumber
\end{eqnarray} 
On the other hand, 
by $\eta_2\times$(\ref{r3})$+\eta_3\times$(\ref{r4}),
\begin{eqnarray*} 
2(\xi_1-\eta_1)(\eta_2\xi_2+\eta_3\xi_3)
-2\xi_1(\eta_2^2+\eta_3^2)
=\eta_2\pt_{\eta_2}\phi+\eta_3\pt_{\eta_3}\phi.
\end{eqnarray*} 
Hence by (\ref{r7}), we obtain
\begin{eqnarray*} 
2(\xi_1-\eta_1)
\left\{
-\frac{3\eta_1^2+\eta_2^2+\eta_3^2}{2\eta_1}\xi_1+N_1
\right\}
-2\xi_1(\eta_2^2+\eta_3^2)
=\eta_2\pt_{\eta_2}\phi+\eta_3\pt_{\eta_3}\phi.
\end{eqnarray*} 
Therefore, 
\begin{eqnarray}
\lefteqn{\frac{3\eta_1^2+\eta_2^2+\eta_3^2}{\eta_1^2}\xi_1^2
-\frac{3\eta_1^2-\eta_2^2-\eta_3^2}{\eta_1}\xi_1}
\qquad\qquad\qquad
\label{r11}\\
&=&
\frac{1}{\eta_1}
\left\{2(\xi_1-\eta_1)N_1-\eta_2\pt_{\eta_2}\phi-\eta_3\pt_{\eta_3}\phi\right\}.
\nonumber
\end{eqnarray} 
By $\displaystyle{(3\eta_1^2+\eta_2^2+\eta_3^2)/\eta_1^2\times}$
(\ref{r12})$-\{9\eta_1^4+18\eta_1^2(\eta_2^2+\eta_3^2)+(\eta_2^2+\eta_3^2)^2\}
/\{4\eta_1^2(\eta_2^2+\eta_3^2)\}\times$(\ref{r11}), we have
\begin{eqnarray*} 
3\frac{(\eta_1^2+\eta_2^2+\eta_3^2)(3\eta_1^2-\eta_2^2-\eta_3^2)^2}
{4\eta_1^3(\eta_2^2+\eta_3^2)}\xi_1
=N_3,
\end{eqnarray*} 
where $N_3=N_3(\xi,\eta)$ is given by 
\begin{eqnarray}
\label{n3}\\
N_3&=&
\frac{3\eta_1^2+\eta_2^2+\eta_3^2}{\eta_1^2}
(\pt_{\eta_1}\phi+2N_1-N_2)
\nonumber\\
& &-\frac{9\eta_1^4+18\eta_1^2(\eta_2^2+\eta_3^2)
+(\eta_2^2+\eta_3^2)^2}{4\eta_1^3(\eta_2^2+\eta_3^2)}
\{2(\xi_1-\eta_1)N_1-\eta_2\pt_{\eta_2}\phi-\eta_3\pt_{\eta_3}\phi\}.
\nonumber
\end{eqnarray}
Hence we obtain  
\begin{eqnarray}
\xi_1
=
\frac{4\eta_1^3(\eta_2^2+\eta_3^2)}
{3(\eta_1^2+\eta_2^2+\eta_3^2)(3\eta_1^2-\eta_2^2-\eta_3^2)^2}
N_3.
\label{key1}
\end{eqnarray}
Substituting (\ref{n1}), (\ref{n2}), (\ref{n3}) 
into (\ref{key1}), we have (\ref{key2}), 
where
\begin{eqnarray*}
A_{0}(\eta)
&=&\frac{1}{3p(\eta)}
\{9\eta_1^4+30\eta_1^2(\eta_2^2+\eta_3^2)+5(\eta_2^2+\eta_3^2)^2\},\\
A_{1}(\eta)
&=&-\frac{1}{3p(\eta)}
\eta_1\{9\eta_1^4+18\eta_1^2(\eta_2^2+\eta_3^2)+(\eta_2^2+\eta_3^2)^2\},\\
A_{j}(\eta)
&=&-\frac{4}{3p(\eta)}
\eta_i
(\eta_2^2+\eta_3^2)(3\eta_1^2+\eta_2^2+\eta_3^2),\quad j=2,3,\\
B_{0,1}(\eta)
&=&-\frac{4}{p(\eta)}\eta_1(\eta_2^2+\eta_3^2),
\\
B_{0,k}(\eta)
&=&-\frac{2}{3p(\eta)}\eta_k(3\eta_1^2+\eta_2^2+\eta_3^2),
\quad k=2,3,\\
B_{1,1}(\eta)
&=&\frac{4}{p(\eta)}\eta_1^2(\eta_2^2+\eta_3^2),\\
B_{1,k}(\eta)
&=&\frac{2}{3p(\eta)}\eta_1\eta_k(3\eta_1^2+\eta_2^2+\eta_3^2),
\quad k=2,3,\\
B_{j,1}(\eta)
&=&\frac{4}{p(\eta)}\eta_1\eta_j(\eta_2^2+\eta_3^2),\quad j=2,3,\\
B_{j,2}(\eta)
&=&\frac{2}{3p(\eta)}(\eta_2^2+\eta_3^2)(3\eta_1^2+\eta_2^2+\eta_3^2),
\quad j=2,3,\\
C_{j,1}(\eta)
&=&-\frac{2}{3p(\eta)}\eta_{5-j}(3\eta_1^2+\eta_2^2+\eta_3^2),\quad j=2,3,\\
C_{j,2}(\eta)
&=&\frac{2}{3p(\eta)}\eta_j(3\eta_1^2+\eta_2^2+\eta_3^2),\quad j=2,3,
\end{eqnarray*}
with $p(\eta)=(\eta_1^2+\eta_2^2+\eta_3^2)(3\eta_1^2-\eta_2^2-\eta_3^2)^2$. 
Hence we have Lemma \ref{algebra}.
\end{proof}

\subsection*{Acknowledgements} 
J.S was supported by JSPS KAKENHI Grant Numbers
JP25H00597 and JP23K20805.


\begin{thebibliography}{99} 

\bibitem{A}
Anjolras P., 
\textit{Scattering of the 2D modified Zakharov-Kuznetsov equation}, 
preprint. available at arXiv:2506.17179. 



\bibitem{CK} Correia S. and Kinoshita S., 
\textit{Global well-posedness and scattering for 
the 2D modified Zakharov-Kuznetsov equation}, 
preprint, available at arXiv:2507.23397. 

\bibitem{CMPS}
C\^{o}te R., Mu\~{n}oz C., Pilod D., and Simpson G., 
\textit{Asymptotic stability of high-dimensional Zakharov-Kuznetsov solitons}, 
Arch. Ration. Mech. Anal., {\bf 220} (2016) 639--710. 

\bibitem{D} de Bouard A.,  
\textit{Stability and instability of some nonlinear dispersive solitary waves 
in higher dimension}, Proc. Roy. Soc. Edinburgh Sect. A, 
{\bf 126} (1996) 89--112.

\bibitem{F1}
Faminskii A. V., 
\textit{The Cauchy problem for the Zakharov-Kuznetsov equation}. 
Differ. Equations, {\bf 31} (1995), 1002--1012.


\bibitem{FHRY1}
Farah L. G., Holmer J., Roudenko S. and Yang K., 
\textit{Asymptotic stability of solitary waves of 
the 3D quadratic Zakharov-Kuznetsov equation}, 
Amer. J. Math. {\bf 145} (2023), 1695--1775.

\bibitem{FLP}
Farah L.G., Linares F. and Pastor A., 
\textit{A note on the 2D generalized Zakharov-Kuznetsov equation: 
local, global, and scattering results}. 
J. Differential Equations {\bf 253} (2012), 2558--2571.

\bibitem{GMS1} Germain P, Masmoudi N. and Shatah J., 
\textit{Global solutions for 3D quadratic Schr\"odinger equations}. 
Int. Math. Res. Not. {\bf 2009} (2009), 414--432.

\bibitem{GMS2} Germain P, Masmoudi N. and Shatah J., 
\textit{Global solutions for 2D quadratic Schr\"odinger equations}. 
J. Math. Pures Appl. (9) {\bf 97} (2012), 505--543.



\bibitem{GH}
Gr\"unrock, A. and Herr S., 
\textit{The Fourier restriction norm method for the Zakharov-Kuznetsov 
equation}. Discrete Contin. Dyn. Syst. {\bf 34} (2014), 2061--2068.

\bibitem{GNT2} Gustafson S., Nakanishi K., and Tsai T.-P.,
\textit{Scattering theory for the Gross-Pitaevskii equation 
in three dimensions}. Commun. Contemp. Math. 
{\bf 11} (2009), 657--707.

\bibitem{HK1}
Herr S. and Kinoshita S., 
\textit{The Zakharov-Kuznetsov equation in high dimensions: 
small initial data of critical regularity}. 
J. Evol. Equ. {\bf 21} (2021), 2105--2121.

\bibitem{HK2}
Herr S. and Kinoshita S., 
\textit{Subcritical well-posedness results for 
the Zakharov-Kuznetsov equation in dimension three and higher}. 
Ann. Inst. Fourier (Grenoble) {\bf 73} (2023), 1203--1267.



\bibitem{KPV}
Kenig C.E., Ponce G. and Vega L., \textit{
Oscillatory integrals and regularity
of dispersive equations}, Indiana Univ.math J. 
{\bf 40} (1991), 33--69.

\bibitem{K}
Kinoshita S., 
\textit{Global well-posedness for the Cauchy problem 
of the Zakharov-Kuznetsov equation in 2D}. 
Ann. Inst. H. Poincar\'e C Anal. Non Lin\'e aire {\bf 38} (2021), 451--505.


\bibitem{Kl} Klainerman S., 
\textit{The null condition and global existence 
to nonlinear wave equations}, 
Lectures in Appl. Math., {\bf 23}
American Mathematical Society, Providence, RI, 
(1986), 293--326.

\bibitem{KSS}
Klein C., Saut J.-C. and Stoilov N., 
\textit{Numerical study of the transverse stability 
of line solitons of the Zakharov-Kuznetsov equations}, 
Phys. D {\bf 448} (2023), Paper No. 133722, 11 pp.


\bibitem{LSp} 
Laedke E. W.  and Spatschek K.-H. , 
\textit{Nonlinear ion-acoustic waves in weak magnetic fields}. 
Phys. Fluids, {\bf 25} (1982), 985--989. 

\bibitem{LLS}
Lannes D., Linares F. and Saut J.-C., 
\textit{The Cauchy problem for the Euler-Poisson system 
and derivation of the Zakharov-Kuznetsov equation}. 
Studies in phase space analysis with applications to PDEs, 
Progr. Nonlinear Differential Equations Appl., {\bf 84} (2013), 181--213.

\bibitem{LS} 
Linares F. and Saut J.-C., \textit{The Cauchy problem 
for the 3D Zakharov-Kuznetsov equation}. 
Discrete Contin. Dyn. Syst., {\bf 24} (2009), 547--565. 

\bibitem{MP}
Molinet L. and Pilod D., 
\textit{Bilinear Strichartz estimates for the Zakharov-Kuznetsov 
equation and applications}. 
Ann. Inst. H. Poincar\'e C Anal. Non Lin\'eaire {\bf 32} (2015), 347--371.

\bibitem{OT}
Ozawa T.  and Tsutsumi Y., 
\textit{Global existence and asymptotic behavior of 
solutions for the Zakharov equations in three space 
dimensions}, Adv. Math. Sci. Appl. {\bf 3} (1993/94), 
Special Issue, 301--334. 

\bibitem{RS} 
Reed M. and Simon B., 
\textit{Methods of modern mathematical physics. III}
Academic Press, Inc. [Harcourt Brace Jovanovich, Publishers], 
New York, (1980)

\bibitem{RV}
Ribaud F. and Vento S., 
\textit{Well-posedness results for the three-dimensional 
Zakharov-Kuznetsov equation}. 
SIAM J. Math. Anal., {\bf 44} (2012), 2289--2304.

\bibitem{S} Segata J., 
\textit{Existence of wave operators for Zakharov-Kuznetsov 
equation in two space dimensions}, 
Discrete and Continuous Dynamical Systems
{\bf 51} (2026) 230--249.

\bibitem{Y}
Yamazaki Y., 
\textit{Center stable manifolds around line solitary waves of 
the Zakharov-Kuznetsov equation}, 
J. Dynam. Differential Equations {\bf 36} (2024), 871--914.

\bibitem{ZK}
Zakharov V. E. and Kuznetsov E. A., 
\textit{Three-dimensional solitons}. 
Sov. Phys. JETP, {\bf 39} (1974), 285--286.  
\end{thebibliography}
\end{document}